\documentclass{article}
\usepackage[utf8x]{inputenc}
\usepackage{amssymb}
\usepackage{amsmath}
\usepackage{mathrsfs}
\usepackage{bm}
\usepackage{amsthm}
\usepackage{enumitem}
\usepackage{color}
\usepackage{graphicx}
\usepackage{bm}
\usepackage{ulem}
\usepackage{upgreek}

\newcommand{\ci}[1]{\mathscr{#1}}
\newcommand{\g}[1]{\mathfrak{#1}}
\renewcommand{\ni}{\nu}
\newcommand{\alfa}{\alpha}
\newcommand{\R}{\mathbf{R}}
\newcommand{\C}{\mathbf{C}}
\newcommand{\per}{\cdot}
\newcommand{\D}{\Delta}
\newcommand{\de}{\partial}
\newcommand{\grad}{\nabla}
\renewcommand{\div}{\operatorname{div}}
\renewcommand{\H}{\mathbf{H}}

\newcommand{\N}[1]{\left\lVert#1\right\rVert}
\newcommand{\e}{\varepsilon}
\newcommand{\bra}{\left\langle}
\newcommand{\ket}{\right\rangle}
\renewcommand{\O}{\Omega}

\DeclareMathOperator{\vol}{vol}

\newtheorem{proposizione}{Proposition}[section]
\newtheorem{teorema}[proposizione]{Theorem}
\newtheorem{lemma}[proposizione]{Lemma}

\title{Singular solutions of the Yamabe problem in the Heisenberg group and their bifurcation}
\author{Claudio Afeltra\footnote{Scuola Normale Superiore, Piazza dei Cavalieri 7, 56126 Pisa (Italy) - claudio.afeltra@sns.it}}
\date{}

\begin{document}

\maketitle

\begin{abstract}
 We prove the existence of a homogeneous singular solution of the critical equation
 $$-\D u = u^{\frac{Q+2}{Q-2}}$$
 on the Heisenberg group $H^n$, where $Q$ is the \textit{homogeneous dimension}. In order to do this, we introduce a suitable concept of normal curvature for hypersurfaces. Furthermore we study
 the bifurcation of non-homogeneous solutions from the homogeneous one.
\end{abstract}

\section{Introduction}
The Yamabe problem has drawn a large interest in Riemannian geometry. Its solution in the compact case (due to the works of Yamabe, Trudinger, Aubin and Schoen)
has constituted a major advance in the fields of geometric analysis and partial differential equations, and it has been drawing attention until today
(on this topic see, for example, \cite{A} for a general treatment).

In the field of CR geometry, the analogous problem of finding a conformal metric with constant Webster curvature (the analogous of scalar curvature).

On the Heisenberg group, which is the ``model'' CR manifold, the Yamabe problem is equivalent to finding the positive solutions of the equation
\begin{equation}\label{Equazione}
 -\D u = u^{\frac{Q+2}{Q-2}},
\end{equation}
where $\D$ is the sublaplacian and $Q$ is the homogeneous dimension (precise definitions are given later).

The positive solutions of this equation satisfying some integrability hypotheses were classified by Jerison and Lee \cite{JL};
geometrically they correspond to conformal factors that trasform the standard pseudohermitian structure of $\H^n$ into the push-forward of the
pseudohermitian structure of the sphere ${\bf S}^{2n+1}\subset\C^{n+1}$ with respect to the Cayley transform, up to translations and dilations.
This classification plays an important role in the solution of the CR Yamabe problem, see \cite{JL2}, \cite{GamYac}, \cite{Gam} and \cite{CMY}.

Additionally to this, it is interesting to study the problem on $\H^n\setminus\{0\}$.
In the Riemannian case all the solutions singular at a point were classified by Caffarelli, Gidas and Spruck (see \cite{CGS}). These form a continuous one-parameter family of radially periodic metrics depending on a parameter $\tau\in(0,1]$, called Delaunay metrics: for $\tau=1$ the metric is homogeneous and corresponds to the
cylindrical metric, and for $\tau\to 0$ tends to a superposition of regular solutions.
This classification has been useful in the study of the profiles of general singular solutions (see \cite{KMPS}), as well as in the study of blow-ups in the problem of
prescribed curvature (see \cite{L1}, \cite{L2}, \cite{CL}).

The author of this article proved (see \cite{Af}) the existence of analogues of the Delaunay metrics for small values of $\tau$, constructed
by perturbing an approximated solution consisting of a series of regular solutions suitable dilated.
In this article, it is proved the existence of a homogeneous solution analogous to the Euclidean one:

\begin{teorema}\label{Teorema}
 There exists a solution $\Psi$ of the equation
 $$-\D \Psi = \Psi^{\frac{Q+2}{Q-2}},$$
 defined on $\H^n\setminus\{0\}$, such that $\Psi\circ\delta_{\lambda}=\lambda^{\frac{Q-2}{2}}\Psi$ and $\Psi(z,t)= \Psi(|z|,t)$.
\end{teorema}

The above result is proved by posing the problem in a variational form, and then performing a conformal
change that trasforms $\H^n\setminus\{0\}$ in a pseudohermitian cylinder, and imposing symmetries in order to reduce the problem to an ODE with variational structure.

The main difficulty is that, because of the non compactness of $\H^n\setminus\{0\}$, the problem has to be
formulated on a closed annulus $\{1\le |x|\le r\}$ (where $|\cdot|$ is the homogeneous norm), and so one has to put boundary conditions that, under a conformal change, behave in a
treatable way.
It is known that the mean curvature behaves in such a way, indeed the prescription
of the mean curvature of the boundary is considered the most natural boundary condition in the prescribed curvature problem for manifolds with boundary
(see, for example, \cite{E}).
In our case there is not such a concept, except in dimension three (see \cite{CHMY}). So we introduce, in arbitrary dimension, the notion of
canonical pseudohermitian normal curvature. In such a way we can formulate variationally the problem of the prescription of the
Webster curvature with boundary conditions, with a functional that is conformally invariant.

In the second part of the article, we study the problem of the bifurcation of radially periodic solutions from the homogeneous one we found.
That is, considering a parameter $T$, we want to prove that solutions to equation \eqref{Equazione} such that $u\circ\delta_T=T^{-\frac{Q-2}{2}}$
bifurcate from the homogeneous solution for infinitely many values of $T$. This problem has a variational structure: the radially periodic solutions are
critical points of a certain functional $\ci{J}_T:X_T\to\R$ whose Morse index tends to infinity for $T\to\infty$, and so, such that
its second differential is singular for infinitely many values of $T$. We prove the following result.

\begin{teorema}\label{Teorema2}
 There exists arbitrarily large values of $T$ for which $d^2\ci{J}_T$ is singular, and every such value is a bifurcation value.
\end{teorema}

The article is structured as follows. After the preliminaries of Section \ref{Notazioni}, in Section \ref{Curvatura} we introduce a notion of
curvature of a hypersurface in a pseudohermitian manifold that serves our purposes. In Section \ref{SoluzioneSingolare} we
prove Theorem \eqref{Teorema}
solve the problem by a conformal change of pseudohermitian metric, and through the imposition of natural simmetries.
In Section \ref{Biforcazione} we study the problem of bifurcation.

\section{Preliminaries and notation}\label{Notazioni}
For a general introduction to CR manifolds we refer to \cite{DT}, but we recall here some basic concepts.

A  \textit{CR manifold} is a real smooth manifold $M$ endowed with a subbundle $\ci{H}$ of the complexified tangent bundle of $M$, $T^{\C}M$, such that $\ci{H}\cap\overline{\ci{H}} = \{0\}$
and $[\ci{H},\ci{H}]\subseteq\ci{H}$. We will assume $M$ to be of hypersurface type, that is that $\dim M=2n+1$ and that $\dim\ci{H}=n$.
There exists a non-zero real differential form $\theta$ that is zero on $\g{Re}(\ci{H}\oplus\overline{\ci{H}})$;
it is unique up to scalar multiple by a function.
Such a form is called  \textit{pseudohermitian} or \textit{contact form}. On a \textit{pseudohermitian manifold}, the  \textit{Levi form} on $\ci{H}$ is defined as the 2-form
$L_{\theta}(V,W)=-id\theta(V,\overline{W})=id\theta([V,\overline{W}])$. A CR manifold is said to be \textit{pseudoconvex} if it admits a positive definite
Levi form (this implies every Levi form to be definite), it is said \textit{nondegenerate} if it admits a nondegenerate Levi form.
An almost complex structure $J$ can be defined on $H(M)= \g{Re}(\ci{H}+\overline{\ci{H}})$ in a natural way, by
$J(V+\overline{V}) = i(V - \overline{V})$.
There exists a unique vector field $T$ such that $\theta(T)=1$ and $i_Td\theta = 0$.
This permits to define a natural Riemannian metric $g_{\theta}$, which coincides with the metric $G_{\theta}$ associated to $L_{\theta}$ on $H(M)$, and such that $g_{\theta}(T,T)=1$ and $T$ is orthogonal to $H(M)$.
 
On a nondegenerate pseudohermitian manifold one can define a connection, the Tanaka-Webster connection. This allows to define
curvature operators in an analogous manner as in Riemannian geometry: the pseudohermitian curvature tensor is the curvature of the
Tanaka-Webster connection, the Ricci tensor is
$$\operatorname{Ric}(X,Y)=\operatorname{trace}(Z\mapsto R(Z,X)Y),$$
and the Webster scalar curvature is the trace of the Ricci tensor with respect to the Levi form.
 
If $\widetilde{\theta}=u^{2/n}\theta$, the transformation law of the Webster curvature is
\begin{equation}\label{TrasformazioneCurvatura}
 \widetilde{W}=u^{-1-2/n}\left(-\frac{2n+2}{n}\D_bu + Wu\right),
\end{equation}
(see \cite{DT}), where $\D_b$ is the sublaplacian, which can be defined as the divergence of the subgradient with respect to the natural volume form in such context,
$\theta\wedge (d\theta)^n$.
So the Yamabe problem leads to the equation
$$-\frac{2n+2}{n}\D_bu + Wu=\lambda u^{1+2/n}.$$

The model pseudohermitian manifold, and the one we will study, is the Heisenberg group $\H^n$,
that is the Lie group $\C^n\times\R$ with the product
$$(z_1,t_1)\per (z_2,t_2)=(z_1+z_2, t_1+t_2+2\,\g{Im}(z_1\per\overline{z_2})),$$
endowed with the subbundle spanned by the standard left invariant vector fields
$$Z_{\alfa}=  \frac{\de}{\de z_{\alfa}}+i\overline{z_{\alfa}}\frac{\de}{\de t},$$
and with the unique left invariant pseudohermitian with respect to this CR structure. In this case the sublaplacian coincides with the
group sublaplacian given by
$$\D = 2 \sum_{\alfa=1}^n (Z_{\alfa}\overline{Z_{\alfa}} + \overline{Z_{\alfa}}Z_{\alfa}).$$
We will use also the Koranyi norm, given by
$$|(z,t)|=\left(|z|^4+t^2\right)^{1/4},$$
and the dilations given by
$$\delta_{\lambda}(z,t)=(\lambda z, \lambda^2t)$$
for $\lambda>0$. Notice that, in the above notation, $|\delta_{\lambda}x|= \lambda |x|$.

It turns out that $\H^n$ has zero Webster curvature, and so, up to an inessential constant, the Yamabe problem is equivalent to find positive solution to equation
$$-\D u = u^{\frac{Q+2}{Q-2}},$$
(where $Q=2n+2$ is the \textit{homogeneous dimension}).

An important transformation in the Heisenberg group is the Kelvin inversion
$$\ci{K}:\H\setminus\{0\}\to \H\setminus\{0\}$$
given by
$$\ci{K}(z,t)= \left( \frac{-iz}{t+i|z|^2}, -\frac{t}{\rho^4}\right).$$
$\ci{K}$ leaves the unit sphere invariant, but, unlike its analogous on the Euclidean space, it does not fix the unit sphere pointwise.

We recall the following formulas from \cite{L} for conformal changes of pseudohermitian metric, to which we refer also for the notation.

\begin{proposizione}
 If $Z_1,\ldots, Z_n$ are \ldots and $\nabla Z_{\alfa} = \omega_{\alfa}^{\beta}\otimes Z_{\beta}$
 and $\theta^{\alfa}...$, then, under the conformal change $\theta\mapsto \widetilde{\theta} = e^{2f}\theta$,
 the Tanaka-Webster connection trasforms as
 $$\widetilde{\omega}_{\alfa}^{\beta} = \omega_{\alfa}^{\beta} + 2(f_{\beta}\theta^{\alfa} - f_{\alfa}\theta^{\beta})+
 \delta_{\alfa}^{\beta}(f_{\gamma}\theta^{\gamma}-f^{\gamma}\theta_{\gamma}) + F\per\theta$$
 (where $F$ is a function of $f$ explicitly known, but whose expression is irrelevant for our purposes).
\end{proposizione}

\section{The canonical pseudohermitian normal curvature}\label{Curvatura}
Let $\Sigma$ be a two-sided hypersurface in $M$ such that $V=\dim(T\Sigma\cap H(M)) = 2n-1$ at every point. If $N$ is a normal vector field to $\Sigma$ with respect to $g_{\theta}$,
the normalization of his orthogonal projection on $H(M)$, $\ni$, is normal to $V$. Let $\xi=-J\ni$. This is a canonical direction (given an orientation
on $\Sigma$). So we define the \textit{canonical pseudohermitian normal curvature} of $\Sigma$ as
$$\kappa = g_{\theta}(\nabla_{\xi} \xi, \ni).$$

\begin{proposizione}
 Under the conformal change $\theta\mapsto \widetilde{\theta}^{2/n} \theta$, $\kappa$ the canonical pseudohermitian normal curvature of
 the new pseudohermitian metric is given by the formula
 $$ \kappa u - \frac{3}{n} \ni(u) = u^{1+\frac{1}{n}} \widetilde{\kappa}.$$
\end{proposizione}

\begin{proof}
Let us choose a frame $Z_1,\ldots,Z_n$ for $\ci{H}$ such that $Z_1+\overline{Z}_1=\xi$ and $Z_2,\ldots,Z_n$ form a frame for $(T\Sigma\otimes\C)\cap \ci{H}$.
Because of formula (4.2) in \cite{L}, $\omega_{\overline{1}}^{\overline{1}}= -\omega_1^1$. Then, since $\ni=i(Z_1-\overline{Z}_1)$,
$$\kappa = g_{\theta}(\nabla_{\xi} \xi, \ni) = ig_{\theta}(\nabla_{\xi}Z_1+\nabla_{\xi}\overline{Z}_1, Z_1-\overline{Z}_1) =$$
$$= ig_{\theta}(\omega_1^{\alfa}(\xi) Z_{\alfa} + \omega_{\overline{1}}^{\overline{\alfa}}(\xi) Z_{\overline{\alfa}}, Z_1-\overline{Z}_1) =$$
$$= -i \omega_1^1(\xi) g_{\theta}(Z_1,\overline{Z}_1) - i\omega_1^1(\xi)g_{\theta}(Z_1,\overline{Z}_1) = -i\omega_1^1(\xi).$$

Applying Lee's formula, since $h_{1\overline{1}}=1$ we obtain that
$$\widetilde{\omega}_1^1 = \omega_1^1 + 2(f_1\theta^1 - f_1\theta^1)+ \delta_1^1(f_1\theta^1-f^1\theta_1) + F \per\theta=$$
$$= \omega_1^1 + 3(Z_1f \theta^1 + \overline{Z}_1f\overline{\theta}^1) \;\;\;\;\;\;\operatorname{mod}\theta.$$
Considering that after the conformal change the Levi form is multiplied by $e^{2f}$, and so the canonical tangent vector becomes
$\widetilde{\xi}= e^{-f}\xi$, we obtain that
$$\widetilde{\kappa} = -i\widetilde{\omega}_1^1(\widetilde{\xi}) = -ie^{-f}(\omega_1^1 + 3(Z_1f \theta^1 - \overline{Z}_1f\overline{\theta}^1))(\xi)=$$
$$= e^{-f}\kappa - 3i(Z_1f \theta^1 - \overline{Z}_1f\overline{\theta}^1))(Z_1+\overline{Z}_1) = e^{-f}\kappa -3i(Z_1-\overline{Z}_1)f =$$
$$= e^{-f}\kappa- 3\ni(f).$$
This concludes the proof.
\end{proof}

Recall that the Webster curvature transforms by the formula
$$ -b_n \D_b u + Wu = \widetilde{W}u^{1+\frac{2}{n}},$$
where $b_n = 2+\frac{2}{n}$.
Now let us pick a frame of $H(M)$, $e_1,\ldots, e_{2n}$ such that $e_{2n-1}= \xi$, $e_{2n} = \ni$, and such that it is orthonormal with respect to the Levi form.
Let $e^1,\ldots,e^{2n}$ be the dual basis thereof.
Then, if $Z_{\alfa}= \frac{1}{2}(e_{2\alfa-1}-ie_{2\alfa})$ and $\overline{Z}_{\alfa} = \frac{1}{2}(e_{2\alfa-1}+ie_{2\alfa})$, the dual basis thereof, $\theta^1,\ldots,\theta^n$ is given by the formula $\theta^{\alfa}=e^{2\alfa-1}+ie^{2\alfa}$.
It holds that
$$\theta^{\alfa}\wedge \overline{\theta^{\alfa}} = (e^{2\alfa-1}+ie^{2\alfa})\wedge(e^{2\alfa-1}-ie^{2\alfa})= -2ie^{2\alfa-1}\wedge e^{2\alfa}.$$
By the definition of the Levi form we have
$$ d\theta = i\sum_{\alfa=1}^n \theta^{2\alfa-1}\wedge \overline{\theta^{2\alfa}} = i\sum_{\alfa=1}^n(e^{2\alfa-1}+ie^{2\alfa})\wedge(e^{2\alfa-1}-ie^{2\alfa})= 2\sum_{\alfa=1}^ne^{2\alfa-1}\wedge e^{2\alfa},$$
so
$$\theta\wedge(d\theta)^n = 2^n\theta\wedge\left(\sum_{\alfa=1}^ne^{2\alfa-1}\wedge e^{2\alfa}\right)^n = 2^nn!\theta\wedge e^1\wedge\ldots\wedge e^n=2^nn!\vol_{g_{\theta}}.$$

We want to give a variational formulation to the problem of the prescription of the Webster curvature and the prescription of the canonical pseudohermitian normal curvature on the boundary.

\begin{proposizione}
 The functional
 $$Q(v) = \int_{M}(b_n|\nabla v|^2 +Wv^2) \theta\wedge(d\theta)^n - c_n\int_{\de M} \kappa v^2 \sigma\wedge\theta,$$
 where $c_n=\frac{b_n}{3}n2^nn!$ and $\sigma = e^1\wedge e^2\wedge\ldots\wedge e^{2n-1}$, is invariant by the transformation
$$\theta\mapsto \widetilde{\theta}= u^{2/n}\theta, \;\;\; v\mapsto \widetilde{v} = vu^{-1}.$$
\end{proposizione}

\begin{proof}
Under this conformal change $G_{\theta}\mapsto u^{2/n}G_{\theta}$, and so $\widetilde{\nabla}=u^{-2/n}\nabla$. Therefore
$$\int_{M}|\widetilde{\nabla} \widetilde{v}|^2 \widetilde{\theta}\wedge(d\widetilde{\theta})^n = \int_{M}u^{-2/n}|\nabla(u^{-1}v)|^2 u^{2(n+1)/n}\theta\wedge(d\theta)^n=$$
$$ = u^2 \int_{M} |u^{-1}\nabla v - u^{-2}v\nabla u|^2 \theta\wedge(d\theta)^n =$$
$$=\int_{M}\left( |\nabla v|^2 + u^{-2}v^2|\nabla u|^2 - 2u^{-1}v\nabla u\per\nabla v\right)\theta\wedge(d\theta)^n=$$
$$= \int_{M} |\nabla v|^2 + \int_{M}\left( v^2|\nabla \log u|^2 - \nabla \log u\per\nabla (v^2)\right)\theta\wedge(d\theta)^n=$$
$$= \int_{M} |\nabla v|^2 + \int_{M}v^2\left( |\nabla \log u|^2 + \D_b\log u\right)\theta\wedge(d\theta)^n - 2^nn!\int_{\de M}v^2 g_{\theta}(\grad\log u,\xi)\ci{V},$$
where $\ci{V}$ is the volume form associated to the restiction of $g_{\theta}$. It is easy to verify that for every $X$ in $H(M)$, the restriction
of $g_{\theta}(\xi,X)\ci{V}$ is equal to the restriction of $e^{2n}(X)\sigma\wedge\theta$. So
$$\int_{M}|\widetilde{\nabla} \widetilde{v}|^2 \widetilde{\theta}\wedge(d\widetilde{\theta})^n =$$
$$= \int_{M} |\nabla v|^2 + \int_{M}v^2\left( |\nabla \log u|^2 + \D_b\log u\right)\theta\wedge(d\theta)^n - 2^nn!\int_{\de M}v^2 \ni(\log u)\sigma\wedge\theta.$$
Thanks to the conformal change formula,
$$ \int_{M} \widetilde{W}\widetilde{v}^2 \widetilde{\theta}\wedge(d\widetilde{\theta})^n =
\int_{M}( -b_nu^{-1-2/n}\D_b u + Wu^{-2/n})v^2u^{-2}u^{2+2/n}\theta\wedge(d\theta)^n =$$
$$ = \int_{M} (-b_n u^{-1}\D_b u + W) v^2 \theta\wedge(d\theta)^n .$$
It holds that
$$\D_b\log u = \div(\grad\log u) = \div\left(\frac{\grad u}{u}\right) = \frac{\D_b u}{u} - \frac{|\grad u|^2}{u^2} = \frac{\D_b u}{u} - |\grad\log u|^2,$$
and so
$$ \int_{M} \widetilde{W}\widetilde{v}^2 \widetilde{\theta}\wedge(d\widetilde{\theta})^n = \int_{M} Wv^2 \theta\wedge(d\theta)^n
-b_n \int_{M} (\D_b\log u +|\grad\log u|^2) v^2 \theta\wedge(d\theta)^n .$$
Finally
$$ \int_{\de M} \widetilde{\kappa} \widetilde{v}^2 \widetilde{\sigma}\wedge\widetilde{\theta} =
\int_{\de M} \left(u^{-1/n}\kappa -\frac{3}{n}u^{-1-1/n}\ni(u)\right)v^2u^{-2}u^{2+1/n}\sigma\wedge\theta =$$
$$= \int_{\de M} \kappa v^2\sigma\wedge\theta - \frac{3}{n}\int_{\de M} \ni(\log u) v^2 \sigma\wedge\theta.$$
By summing the above identities we get the desired result.
\end{proof}

One can easily check the following Proposition

\begin{proposizione}\label{formulazione}
 A conformal change has Webster curvature $W_1$ and canonical pseudohermitian normal curvature $\kappa_1$ if and only if it is a stationary point of
the functional
$$I_{W_1,\kappa_1}(v) = Q(v) - \frac{n}{n+1}\int_M W_1v^{2+2/n}\theta\wedge(d\theta)^n + \frac{nc_n}{2n+1}\int_{\de M} \kappa_1 v^{2+1/n}\sigma\wedge\theta,$$
that is invariant for the same transformation of $Q$.
\end{proposizione}
 
 \section{Proof of Theorem \ref{Teorema}}\label{SoluzioneSingolare}
 Now that we have a variational and conformally covariant formulation of the problem of prescribed curvature with boundary conditions, thanks to Proposition \ref{formulazione}, we
 study this problem on suitable annuli, imposing that the boundary has zero curvature, a natural condition because of the simmetry given by the Cayley transform.
 So let us study the problem
 $$\begin{cases}
    -b_n\D_b u = u^{1+2/n} \;\;\;\; \text{on} \;\;\;\; A_r\\
    -\frac{3}{n} \ni(u) + \kappa_{A_r}u = 0 \;\;\;\; \text{on} \;\;\;\; \de A_r
   \end{cases}$$
 where $A_r= B_r\setminus \overline{B}_1$, and $B_r=B_r(0)$ with respect to the Koranyi norm.
 The latter problem is equivalent to find the critical points of
 $$I(v)= b_n\int_{M}|\nabla v|^2 \theta\wedge(d\theta)^n - c_n\int_{\de M} \kappa v^2 \sigma\wedge\theta - \frac{n}{n+1}\int_M v^{2+2/n}\theta\wedge(d\theta)^n.$$
 We restrict the functional to functions such that $u(x,t)=u(|x|,t)$.
 
 Let us consider the conformal change
 $$\theta\mapsto \widetilde{\theta} = \rho^{-2}\theta,$$
 where $\rho=|x|$.
 
 \begin{lemma}
  The Webster curvature of $\widetilde{\theta}$ is
  $$\widetilde{W}= -b_nu^{-1-\frac{2}{n}}\D_b u = -b_n\rho^{n+2}\D_b(\rho^{-n})=b_nn^2 \frac{|x|^2}{\rho^2},$$
  and the mean curvature of the boundary of $A_r$ is zero.
 \end{lemma}

 \begin{proof}
 We have
 $$X_{\alfa}(\rho^4)= \left(\frac{\de}{\de x^{\alfa}}+2y^{\alfa}\frac{\de}{\de t}\right)(|x|^4 +t^2) =$$
 $$= 4(x_{\alfa}^3+(|x|^2-x_{\alfa}^2)x_{\alfa}+y_{\alfa}t) = 4( |x|^2x_{\alfa} +y_{\alfa}t);$$
 
 $$Y_{\alfa}(\rho^4) =\left(\frac{\de}{\de y^{\alfa}}-2x^{\alfa}\frac{\de}{\de t}\right)(|x|^4 +t^2) = 4(|x|^2y_{\alfa}-x_{\alfa}t);$$
 
 $$X_{\alfa}^2(\rho^4) = 4\left(\frac{\de}{\de x^{\alfa}}+2y^{\alfa}\frac{\de}{\de t}\right)(x_{\alfa}^3+|y_{\alfa}|^2x_{\alfa}+y_{\alfa}t)
 = 4(|x|^2 + 2|x_{\alfa}|^2 + 2|y_{\alfa}|^2);$$
 
 $$ Y_{\alfa}^2(\rho^4) = 4\left(\frac{\de}{\de y^{\alfa}}-2x^{\alfa}\frac{\de}{\de t}\right)(y_{\alfa}^3+|x_{\alfa}|^2y_{\alfa}-x_{\alfa}t)=
 4(|x|^2 + 2|x_{\alfa}|^2 + 2|y_{\alfa}|^2).$$
 
 $$X_{\alfa}^2(\rho^{-n}) = X_{\alfa}(X_{\alfa}((\rho^4)^{-n/4})) = -\frac{n}{4} X_{\alfa}(\rho^{-n-4} X_{\alfa}(\rho^4))= $$
 $$= \frac{n(n+4)}{16}\rho^{-n-8} |X_{\alfa}(\rho^4)|^2 -\frac{n}{4}\rho^{-n-4} X^2_{\alfa}(\rho^4) = $$
 $$ = \frac{n(n+4)}{16}\rho^{-n-8} 16( |x|^2x_{\alfa} +y_{\alfa}t)^2 -\frac{n}{4}\rho^{-n-4}4(|x|^2 + 2|x_{\alfa}|^2 + 2|y_{\alfa}|^2)=$$
 $$ = n(n+4)\rho^{-n-8} (|x|^2x_{\alfa} +y_{\alfa}t)^2 -n\rho^{-n-4}(|x|^2 + 2|x_{\alfa}|^2 + 2|y_{\alfa}|^2),$$
 and analogously
 $$Y_{\alfa}^2(\rho^{-n}) = n(n+4)\rho^{-n-8} (|x|^2y_{\alfa} -x_{\alfa}t)^2 -n\rho^{-n-4}(|x|^2 + 2|x_{\alfa}|^2 + 2|y_{\alfa}|^2),$$
 so
 $$ \D_b (\rho^{-n}) = \sum_{\alfa=1}^n (X_{\alfa}^2+Y_{\alfa}^2)(\rho^{-n}) = $$
 $$= n(n+4)\rho^{-n-8}(|x|^6 + |x|^2t^2) -2n(n+2)\rho^{-n-4}|x|^2 = -n^2 \rho^{-n-4}|x|^2.$$
 
Since $u=\rho^{-n}$, by formula \eqref{TrasformazioneCurvatura} we get the desired result.

It can be readily verified that the Kelvin transform is isopseudohermitian with respect to $\widetilde{\theta}$
(that is, it preserves the pseudohermitian structure). Also the transformations of $\H^n$ of the form $(z,t)\mapsto(Az,t)$ with $A$ unitary,
and the dilations, are isopseudohermitian. So, for every point $x$ of $\de A_r$, there is a isopseudohermitian transformation that
fixes $x$, leaves its component of $\de A_r$ invariant, but reverses the orientation. Since reversing the orientation changes sign to
$\widetilde{\kappa}$, it follows that $\widetilde{\kappa}=0$.

Since the Cayley transform is isopseudohermitian with respect to $\widetilde{\theta}$ (that is, it preserves the pseudohermitian structure),
and that the transformations of $\H^n$ of the form $(z,t)\mapsto(Az,t)$ with $A$ unitary also are, it can be proved by simmetry that $\widetilde{\kappa}=0$.
\end{proof}

Therefore, thanks to Proposition \ref{formulazione},
$$\widetilde{I}(v) = b_n\int_{A_r}\left(|\widetilde{\nabla}v|_{\widetilde{\theta}}^2 + n^2 \frac{|x|^2}{\rho^2}v^2\right)\widetilde{\theta}\wedge(d\widetilde{\theta})^n-
\frac{n}{n+1}\int_{A_r} v^{2+2/n}\widetilde{\theta}\wedge(d\widetilde{\theta})^n.$$

We want to impose that the solution is homogeneous and symmetric, in the sense that $u\circ\delta_{\lambda}=\lambda^{\frac{Q-2}{2}}u$ and $u(x,t)=u(|x|,t)$.

We want to express this functional in suitable coordinates.

\begin{lemma}\label{coordinate}
 If $v=v(|x|,t)$, in the coordinates $l=\frac{1}{n}\log\rho\in\R$, $\uptau =t/\rho^2\in[-1,1]$ and $\gamma = x/|x|\in {\bf S}^{2n-1}$,
 it holds that
 $$|\widetilde{\nabla}v|_{\widetilde{\theta}}^2 = (1-\uptau^2)^{3/2}\left|\frac{\de v}{\de\uptau}\right|^2+ \frac{1}{4n^2}(1-\uptau^2)^{1/2}\left|\frac{\de v}{\de l}\right|^2.$$
\end{lemma}

\begin{proof} A transformation of $\H^n$ of the form $(x,t)\mapsto(Ax,t)$ with $A$ linear is an isomorphism of the pseudohermitian structure if and only if $A$ is unitary.
Since this kind of transformations preserves the sphere of unit radius, and since the action of the unitary group is transitive between vectors of the same length,
we can calculate $|\widetilde{\nabla}v|_{\widetilde{\theta}}^2$ in the points of the curve
$$(\sqrt[4]{1-t^2}, 0, \ldots, 0, t).$$
At such points
$$X_{\alfa}= \frac{\de}{\de x_{\alfa}}$$
for every $\alfa=1,\ldots,n$,
$$Y_{\alfa} = \frac{\de}{\de y_{\alfa}}$$
for every $\alfa\ne 1$, and
$$Y_1 = \frac{\de}{\de y_1} - 2 \sqrt[4]{1-t^2}\frac{\de}{\de t}.$$
Using the symmetry of $v$ in $x$, we get
$$|\widetilde{\nabla}v|_{\widetilde{\theta}}^2 = v^{-2/n} |\nabla v|_{\theta}^2 = \frac{1}{4}|X_1u|^2 + \frac{1}{4}|Y_1u|^2  =$$
$$ = \frac{1}{4}\left|\frac{\de v}{\de x_1}\right|^2 + (1-t)^{1/2}\left|\frac{\de v}{\de t}\right|^2 .$$
Since
$$\frac{\de\uptau}{\de x_1} = -\frac{1}{2}\uptau^3\frac{\de(\uptau^{-2})}{\de x_1} = -\frac{1}{2}\frac{\uptau^3}{t^2}\frac{\de}{\de x_1}(x_1^4+t^2)=-2tx_1^3;$$
$$\frac{\de l}{\de x_1} = \frac{1}{n}\frac{\de \log\rho}{\de x_1}= \frac{1}{4n}\frac{\de\rho^4}{\de x_1}= \frac{1}{n}x_1^3;$$
$$\frac{\de\uptau}{\de t} = \frac{1}{2}\uptau^{-1}\frac{\de(\uptau^2)}{\de t} = \frac{1}{2}t^{-1}(2t-2t^3)=(1-t^2);$$
$$\frac{\de l}{\de t} = \frac{1}{4n}\frac{\de\rho^4}{\de t}=\frac{1}{2n}t,$$
we obtain
$$|\widetilde{\nabla}v|_{\widetilde{\theta}}^2 = \frac{1}{4}\left|\frac{\de v}{\de x_1}\right|^2 + (1-t^2)^{1/2}\left|\frac{\de v}{\de t}\right|^2 =$$
$$ =  \frac{1}{4}\left|-2tx_1^3\frac{\de v}{\de\uptau} + \frac{x_1^3}{n} \frac{\de v}{\de l}\right|^2 +
(1-t^2)^{1/2}\left||x|^4\frac{\de v}{\de\uptau}+\frac{t}{2n}\frac{\de v}{\de l}\right|^2 =$$
$$= t^2(1-t^2)^{3/2}\left|\frac{\de v}{\de\uptau}\right|^2+ \frac{1}{4n^2}(1-t^2)^{3/2} \left|\frac{\de v}{\de l}\right|^2-\frac{1}{n}t(1-t^2)^{3/2}\frac{\de v}{\de\uptau}\frac{\de v}{\de l}+ $$
$$+ (1-t^2)^{5/2}\left|\frac{\de v}{\de\uptau}\right|^2 + \frac{1}{4n^2}t^2(1-t^2)^{1/2}\left|\frac{\de v}{\de l}\right|^2 + \frac{1}{n}(1-t^2)^{3/2}t\frac{\de v}{\de\uptau}\frac{\de v}{\de l} =$$
$$= (1-t^2)^{3/2}\left|\frac{\de v}{\de\uptau}\right|^2+ \frac{1}{4n^2}(1-t^2)^{1/2}\left|\frac{\de v}{\de l}\right|^2.$$
By the dilation invariance of $\widetilde{\theta}$ we obtain the formula in the general case
\end{proof}

Now we compute the volume form.

\begin{lemma}\label{FormaDiVolume}
 In the coordinates of Lemma \ref{coordinate}
 $$\widetilde{\theta}\wedge(d\widetilde{\theta})^n = 2^nn!(1-\uptau^2)^{(n-2)/2}dl\wedge d\gamma\wedge d\uptau.$$
\end{lemma}

\begin{proof} The volume form becomes
$$\widetilde{\theta}\wedge(d\widetilde{\theta})^n = \rho^{-2(n+1)} \theta\wedge(d \theta)^n = \frac{2^nn!}{\rho^{2(n+1)}}\vol_{g_{\theta}}=\frac{2^nn!}{\rho^{2(n+1)}}|x|^{2n-1}d|x|\wedge d\gamma\wedge dt.$$
By an easy computation
$$ d|x| = \frac{1}{4|x|^3} d(\rho^4-t^2)=  \frac{1}{4|x|^3}d(e^{4ln}(1-\uptau^2))=$$
$$=\frac{1}{4(1-\uptau^2)^{3/4}e^{3nl}}e^{4ln}(4n(1-\uptau^2)dl - 2\uptau d\uptau)=$$
$$= e^{ln}\left(n(1-\uptau^2)^{1/4}dl - \frac{\uptau}{2(1-\uptau^2)^{3/4}}d\uptau\right);$$
$$ dt = d(e^{2nl}\uptau) = e^{2nl}(2n\uptau dl + d\uptau);$$
$$d|x|\wedge dt = e^{3nl}\frac {1}{(1-\uptau^2)^{3/4}} dl\wedge d\uptau,$$
so
$$\widetilde{\theta}\wedge(d\widetilde{\theta})^n = -\frac{2^nn!}{\rho^{2(n+1)}}(1-\uptau^2)^{(2n-1)/4}\rho^{2n-1}d|x|\wedge dt\wedge d\gamma=$$
$$= -\frac{2^nn!}{e^{3nl}}(1-\uptau^2)^{(2n-1)/4}e^{3nl}\frac {1}{(1-\uptau^2)^{3/4}}dl\wedge d\uptau\wedge d\gamma=$$
$$=2^nn!(1-\uptau^2)^{(n-2)/2}dl\wedge d\gamma\wedge d\uptau,$$
as desired.
\end{proof}

Using Lemmas \ref{coordinate} and \ref{FormaDiVolume} we have that
$$\widetilde{I}(v) = b_n\int_{A_r}\left((1-\uptau^2)^{3/2}\left|\frac{\de v}{\de\uptau}\right|^2+ \frac{1}{4n^2}(1-\uptau^2)^{1/2}\left|\frac{\de v}{\de l}\right|^2v^2 +\right.$$
$$ \left.+ n^2 (1-\uptau^2)^{1/2}\right)2^nn!(1-\uptau^2)^{(n-2)/2}dl\wedge d\gamma\wedge d\uptau+$$
$$- \frac{n}{n+1}\int_{A_r} v^{2+2/n}2^nn!(1-\uptau^2)^{(n-2)/2}dl\wedge d\gamma\wedge d\uptau =$$
$$=b_n2^nn!\int_0^{\frac{\log r}{n}}\int_{-1}^1\left((1-\uptau^2)^{(n+1)/2}\left|\frac{\de v}{\de\uptau}\right|^2 + \frac{1}{4n^2}(1-\uptau^2)^{(n-1)/2}\left|\frac{\de v}{\de l}\right|^2v^2+ \right. $$
$$ \left. + n^2 (1-\uptau^2)^{(n-1)/2}\right)dl\wedge d\uptau- \frac{n2^nn!}{n+1}\int_0^{\frac{\log r}{n}}\int_{-1}^1 v^{2+2/n}(1-\uptau^2)^{(n-2)/2}dl\wedge d\uptau.$$
If $\uptau=\sin s$, then
$$\widetilde{I}(v) = b_n2^nn!\int_0^{\frac{\log r}{n}}\int_{-\frac{\pi}{2}}^{\frac{\pi}{2}} \left(\frac{(\cos s)^{n+1}}{(\cos s)^2}\left|\frac{\de v}{\de s}\right|^2 + \frac{1}{4n^2}(\cos s)^{n-1}\left|\frac{\de v}{\de l}\right|^2+ \right. $$
$$ \left.+ n^2 (\cos s)^{n-1}v^2\right)dl (\cos s)ds- \frac{n2^nn!}{n+1}\int_0^{\frac{\log r}{n}}\int_{-\frac{\pi}{2}}^{\frac{\pi}{2}} v^{2+2/n}(\cos s)^{n-2}dl(\cos s)ds =$$
$$= b_n2^nn!\int_0^{\frac{\log r}{n}}\int_{-\frac{\pi}{2}}^{\frac{\pi}{2}}(\cos s)^n \left(\left|\frac{\de v}{\de s}\right|^2+ \frac{1}{4n^2}\left|\frac{\de v}{\de l}\right|^2+n^2v^2 \right)dlds+$$
$$-\frac{n2^nn!}{n+1}\int_0^{\frac{\log r}{n}}\int_{-\frac{\pi}{2}}^{\frac{\pi}{2}} v^{2+2/n}(\cos s)^{n-1}dlds.$$
Now let us look for homogeneous solutions. Homogeous solutions in the original setting correspond to solutions invariant by
translation (in the $l$ direction), and so let us set $\frac{\de v}{\de l}=0$, and $v=v(s)$. In this special case we have
$$\widetilde{I}(v)= b_n2^nn!\frac{\log r}{n}\int_{-\frac{\pi}{2}}^{\frac{\pi}{2}}(\cos s)^n \left((v')^2+ n^2v^2 \right)ds+$$
$$-\frac{n2^nn!}{n+1}\frac{\log r}{n}\int_{-\frac{\pi}{2}}^{\frac{\pi}{2}} v^{2+2/n}(\cos s)^{n-1}ds.$$
The Euler-Lagrange equation for this functional is
$$- \frac{d}{ds}((\cos s)^nv'(s)) + n^2(\cos s)^nv(s)= \frac{n}{2(n+1)}(\cos s)^{n-1}v(s)^{1+2/n},$$
or equivalently
$$- \cos s v''(s) + n \sin s v'(s) +n^2 \cos s v(s) =  \frac{n}{2(n+1)}v(s)^{1+2/n},$$
on the interval $\left(-\frac{\pi}{2},\frac{\pi}{2}\right)$,
with Neumann boundary conditions, that is also the Euler-Lagrange equation (up to rescaling, thanks to homogeneity) of
$$J(v)=\frac{\int_{-\frac{\pi}{2}}^{\frac{\pi}{2}}(\cos s)^n \left((v')^2+ n^2v^2 \right)ds}{\int_{-\frac{\pi}{2}}^{\frac{\pi}{2}} v^{2+2/n}(\cos s)^{n-1}ds}.$$
Let us define the weighted Sobolev and Lebesgue spaces
$$X=\left\{ u\in H^1_{\rm{loc}}\left(-\frac{\pi}{2},\frac{\pi}{2}\right) \;\;\middle| \;\; \int_{-\frac{\pi}{2}}^{\frac{\pi}{2}}(\cos s)^n \left((v')^2+v^2 \right)ds<\infty  \right\},$$
$$ Y = \left\{ u\in L^1_{\rm{loc}}\left(-\frac{\pi}{2},\frac{\pi}{2}\right) \;\;\middle| \;\; \int_{-\frac{\pi}{2}}^{\frac{\pi}{2}}(\cos s)^{n-1} v^{2+2/n}ds<\infty  \right\}.$$

\begin{proposizione}
 $X$ embeds compactly in $Y$.
\end{proposizione}

\begin{proof}
Let $Z$ be the subspace of $H^1(S^{n+1})$ formed by functions invariant by rotation around the last coordinate axis.
Every such function is of the form $v(x) = u(\cos x^{n+2})$, and it is easy to verify thar under such an identification $$\N{v}_Z  = \N{u}_X.$$

So this is an isometric isomorphism between X and Z. By Rellich-Kondrachov's theorem, Z embeds compactly into $L^p(S^n)$ for every $p\in[1, 2\frac{n+1}{n−1})$,
which by similar arguments is isometrically isomorphic to $L^p\left((\left( -\frac{\pi}{2},\frac{\pi}{2}\right),(\cos s)^nds\right)$.
Given $\alfa > 0$, $q > 1$,
$$\int_{-\frac{\pi}{2}}^{-\frac{\pi}{2}} (\cos s)^{n−1}v^{2+2/n}ds = \int_{-\frac{\pi}{2}}^{-\frac{\pi}{2}} \frac{(\cos s)^{n−1+\alfa}}{(\cos s)^{\alfa}}ds\le$$
$$\le \left(\int_{-\frac{\pi}{2}}^{-\frac{\pi}{2}} v^{(2+2/n)q} (\cos s)(n−1+\alfa)q ds\right)^{1/q}\left(\int_{-\frac{\pi}{2}}^{-\frac{\pi}{2}} (\cos s)^{−\alfa q'}ds\right)^{1/q'}.$$
If we impose that $(n − 1 + \alfa)q = n$, then taking $\alfa$ small enough, we can find
that $p = \left(2 + \frac{2}{n}q\right) < 2\frac{n+1}{n−1}$ and $\alfa q' < n$, getting that
$$\int_{-\frac{\pi}{2}}^{-\frac{\pi}{2}} (\cos s)^{n−1}v^{2+2/n}ds \le  C \int_{-\frac{\pi}{2}}^{-\frac{\pi}{2}}\left( v^p (\cos s)^nds\right)^{1/q},$$
that is
$$\N{v}_Y \le C \N{v}_{L^p((\cos s)^n ds)},$$
and so $L^p\left(\left(\frac{\pi}{2},\frac{\pi}{2}\right) (\cos s)^n ds\right)$ embeds into $Y$.
So we get the thesis.
\end{proof}

Now Theorem \ref{Teorema} can be proved, in a standard way, by the direct methods of the calculus of variation.
Since the solution does not depend on $r$, by homogeneity this defines a solution on the whole $\H^n\setminus\{0\}$.

\section{Proof of Theorem \ref{Teorema2}}\label{Biforcazione}
Let $\O_t=\{1\le|x|\le T\}$ be a cylinder in the Heisenberg group.
In the following, all integrals are meant with respect to the Haar measure and volume elements will be omitted.

Let
$$\ci{J}_T(u)=\int_{\Omega_T}\left(|\grad_{\H^n} u|^2-\frac{1}{2^*}|u|^{2^*}\right)$$
be defined on the space
$$X_T=\left\{u\in S^1_{\rm{loc}}(\H^n) \;|\; u\circ \delta_T = T^{-\frac{Q-2}{2}}u\right\},$$
where $S^1_{\rm{loc}}(\H^n)$ is the Stein-Folland space (see \cite{Folland}).
Let $\Psi$ be the homogeneous solution of the PDE
$$-\D u = u^{2^*-1}$$
found in the previous section, that is, a stationary point of $\ci{J}_T$.
 
\begin{proposizione}\label{IndiceDiMorse}
 The Morse index of $\ci{J}_T$ at $\Psi$ is finite and tends to infinity as $T\to\infty$.
\end{proposizione}

\begin{proof}
 The Morse index is finite because the operator on $X_T$ associated to the bilinear form
 $$d^2\ci{J}_T(\Psi)[u,v] = \int_{\Omega_T}\left(\grad u\grad v -(2^*-1)\Psi^{2^*-2}uv\right)$$
is of the sum of the identity and a compact operator (thanks to the Rellich-Kondrachov theorem for the Stein-Folland space).
 
 The signature of a simmetric bilinear form remains invariant passing to the complexification and extending it to a hermitian form.
 So let us take
 $$u(x)=\exp\left(i\frac{1}{M}\log|x|\right)\Psi(x),$$
 where 
 $$\frac{\log T}{2\pi M} \in \mathbf{Z}.$$
 Then
 $$d^2\ci{J}_T(\Psi)[u,u] = \int_{\Omega_T}\left(\left|\grad\left(\exp\left(i\frac{1}{M}\log|x|\right)\Psi\right)\right|^2-(2^*-1)\Psi^{2^*-2}\Psi^2\right)=$$
 $$= \int_{\Omega_T}\left(\frac{1}{M^2|x|^2}\Psi^2 +|\grad\Psi|^2 +2\frac{1}{M|x|}\Psi\grad\Psi\per\grad|x| -(2^*-1)\Psi^{2^*}\right)=$$
 $$= \int_{\Omega_T}\left(\frac{1}{M^2|x|^2}\Psi^2 +2\frac{1}{M|x|}\Psi\grad\Psi\per\grad|x| -(2^*-2)\Psi^{2^*}\right)=$$
 $$= -(2^*-2)\int_{\Omega_T}\Psi^{2^*} + \int_{\Omega_T}\left(\frac{1}{M^2|x|^2}\Psi^2 +2\frac{1}{M|x|}\Psi\grad\Psi\per\grad|x|\right).$$
 By homogeneity the three integrals
 $$\int_{\Omega_T}\Psi^{2^*}, \;\; \int_{\Omega_T}\frac{1}{|x|^2}\Psi^2, \;\; \int_{\Omega_T}\frac{1}{|x|}\Psi\grad\Psi\per\grad|x|$$
 are constant multiples of $\log T$, so there exists a constant $C$ such that if $M\ge C$ then
 $d^2\ci{J}_T(\Psi)[u,u]$ is negative.
 Given $k\in\mathbf{N}$, let $T$ be big enough so that
 $$\frac{2\pi k}{\log T} \le \frac{1}{C}.$$
 Then the functions
 $$u_m(x)=\exp\left(i\frac{2\pi m}{\log T}\log|x|\right)\Psi(x)$$
 with $m=1,\ldots,k$, are such that
 $d^2\ci{J}_T(\Psi)[u_m,u_m] \le -\e \log T$ is negative.
 If $f$ is a homogeneous function of degree zero and $m\ne 0$ then
 $$\int_{\Omega_T}\exp\left(i\frac{2\pi m}{\log T}\log|x|\right)\frac{f(x)}{|x|^Q} = \int_{S_1}d\sigma f\int_1^Tdr\frac{\exp\left(i\frac{2\pi m}{\log T}\log r\right)}{r} = 0.$$
 When calculating
 $$d^2\ci{J}_T(\Psi)[u_m,u_j]$$
 with $m\ne j$, the result is a sum of terms of this kind, so it is zero.
 So the functions $u_m$ span a vector space of dimension $k$ on which $d^2\ci{J}_T(\Psi)$ is negative definite.
\end{proof}

In order to apply bifurcation theory, let us rewrite the functional with respect to the pseudohermitian form given by the conformal change corresponding to $\Psi$.
So we get the functional
$$\widetilde{\ci{J}}_T(u)=\int_{\Omega_T}\left(|\widetilde{\grad}_{\H^n} u|^2+\frac{1}{2}u^2-\frac{1}{2^*}|u|^{2^*}\right),$$
defined on the space
$$Y_T=\left\{u\in S^1_{\rm{loc}}(\H^n) \;|\; u\circ \delta_T = u, \;\; \int_{\O_T} u = 0\right\}.$$
Let $\Sigma$ be the sphere with respect to the Euclidean metric
\footnote{this is necessary to perform the next steps of the proofs because the sphere with respect to the Heisenberg metric is not smooth}.
Let $\phi_k$ be a complete set in $L^2(\Sigma)$ consisting of analytic functions. So
$$\gamma_{k,m,T}(x) = \phi_k\left(\frac{x}{|x|_{\rm{eucl}}}\right)\sin\left(i\frac{2\pi m}{\log T}\log|x|_{\rm{eucl}}\right)$$
is a complete set of functions in $H^1(\O_T)$, analytic with respect to the couple $(x,T)$. SO it is complete also in $S^1(\O_T)$.
With the Gram-Schmidt algorithm, we can obtain a family of Hilbert bases $\psi_{k,T}$ of $S^1(\O_T)$, and preserve the analyticity property.
Let us define the isometry $\Psi_T$ between $Y_T$ and $Y_2$ obtained sending $\psi_{k,T}$ into $\psi_{k,2}$.
Let us call
$$L_T  = \Psi_T\circ \widetilde{\ci{J}}''_T(1) \circ \Psi_T^{-1}.$$
Then, for every $l,k$,
$$\bra L_T u_{2,k}, u_{2,k}\ket_{Y_2} = \bra\widetilde{\ci{J}}''_T(1) u_{k,T}, u_{l,T}\ket_{Y_T} =$$
$$= \int_{\O_T}\widetilde{\grad} u_{k,T}\widetilde{\grad} u_{l,T}+ u_{k,T}u_{l,T} - (2^*-1)u_{k,T}u_{l,T}.$$

It is an analytic function by the following lemma of immediate proof.

\begin{lemma}
 If $S:I\to X$, $T: I\to Y$ are two analytic vector valued functions and $L:X\times Y\to Z$ is a bilinear continuous form, then
 $t\mapsto B(S(t),T(t))$ is analytic.
\end{lemma}

So $L_T$ is an analytic operator-valued function.

It holds that
$$\bra\widetilde{\ci{J}}''_T(1) u,v\ket = \int_{\O_T}\widetilde{\grad} u\widetilde{\grad} v+ uv - (2^*-1)uv = \int_{\O_T}\widetilde{\grad} u\widetilde{\grad} v - (2^*-2)\D(G_Tu)v=$$
$$= \int_{\O_T}\widetilde{\grad} u\widetilde{\grad} v + (2^*-2)\widetilde{\grad}(G_Tu)\widetilde{\grad} v,$$
where $G_T:Y_T\to Y_T$ is the Green's operator, so
$$\widetilde{\ci{J}}''_T(1)= I + (2^*-2)G_T.$$
Since $L_T$ is, by definition, conjugated to $\widetilde{\ci{J}}''_T(1)$, it is of the form $I-K(T)$, where $K(T)$ is an analytic operator-valued function of compact operators.

Now, by means of known results in bifurcation theory, we can prove Theorem \ref{Teorema2}.

\begin{proof}[Proof of Theorem \ref{Teorema2}]
 It suffices to apply Theorem 8.9 in \cite{MW}. In our case the hypotheses of that Theorem are all either trivial or standardly verifiable, with exception
of hypothesis $\gamma$, that is consequence of Corollary 8.3 in the same book.
\end{proof}

\end{document}